\documentclass{amsart}
  
  \usepackage{amsmath,amssymb,amsthm}
 \usepackage{graphicx}
  \usepackage{xspace}
  \usepackage[all]{xy}
  \usepackage{pinlabel}
  \usepackage{enumerate}
  \usepackage{color}

  \usepackage{hyperref}

  \newcommand{\calT}{\mathcal{T}}

  \newcommand{\NN}{\mathbb{N}}

  \newcommand{\QQ}{\mathbb{Q}}
  \newcommand{\RR}{\mathbb{R}}

  \newcommand{\ZZ}{\mathbb{Z}}

  \newcommand{\gothic}{\mathfrak}
  \newcommand{\Ga}{{\gothic a}}
  \newcommand{\Gb}{{\gothic b}}

  \newtheorem{theorem}{Theorem}[section]
  
  \newtheorem{corollary}[theorem]{Corollary}
  \newtheorem{lemma}[theorem]{Lemma}

  \theoremstyle{definition}
  
  \newtheorem{claim}[theorem]{Claim}
  
  \newtheorem*{claim*}{Claim}

  \newtheorem*{question*}{Question}
  \newtheorem*{answer*}{Answer}
  \newtheorem*{application*}{Application}

  \theoremstyle{remark}
  \newtheorem{remark}[theorem]{Remark}
  \newtheorem*{remark*}{Remark}

  \newcommand{\thmref}[1]{Theorem~\ref{Thm:#1}}
  \newcommand{\corref}[1]{Corollary~\ref{Cor:#1}}
  \newcommand{\lemref}[1]{Lemma~\ref{Lem:#1}}

  \newcommand{\figref}[1]{Figure~\ref{Fig:#1}}

  \newcommand{\eqnref}[1]{Equation~\eqref{Eq:#1}}

  \DeclareMathOperator{\twist}{twist}
  \DeclareMathOperator{\width}{width}

  \DeclareMathOperator{\Mod}{Mod}
  \DeclareMathOperator{\Ext}{Ext}
  \DeclareMathOperator{\Hyp}{Hyp}
  
  \DeclareMathOperator{\area}{area}

  \DeclareMathOperator{\I}{i}

    \DeclareMathOperator{\cyl}{cyl}
\DeclareMathOperator{\arcsinh}{arcsinh}

  \newcommand{\emul}{\stackrel{{}_\ast}{\asymp}}
  \newcommand{\gmul}{\stackrel{{}_\ast}{\succ}}
  \newcommand{\lmul}{\stackrel{{}_\ast}{\prec}}
  \newcommand{\eadd}{\stackrel{{}_+}{\asymp}}
  \newcommand{\gadd}{\stackrel{{}_+}{\succ}}
  \newcommand{\ladd}{\stackrel{{}_+}{\prec}}

  \newcommand{\PMF}{\ensuremath{\mathcal{PMF}}\xspace} 
   
  \newcommand{\Teich}{{Teichm\"uller }} 
  \newcommand{\param}{{\mathchoice{\mkern1mu\mbox{\raise2.2pt\hbox{$
  \centerdot$}}
  \mkern1mu}{\mkern1mu\mbox{\raise2.2pt\hbox{$\centerdot$}}\mkern1mu}{
  \mkern1.5mu\centerdot\mkern1.5mu}{\mkern1.5mu\centerdot\mkern1.5mu}}}

  \renewcommand{\setminus}{{\smallsetminus}}

  \newcommand{\ST}{\mathbin{\Big|}} 
   
     \newcommand{\bs}{{\overline s}}  

  \newcommand{\us}{{\underline s}}
  \numberwithin{equation}{section}

\begin{document}


  \title[Teichm\"uller geodesics with $d$-dimensional limit sets]   
  {Teichm\"uller geodesics with $d$-dimensional limit sets}
  
 
  \author{Anna Lenzhen}
  \address{Laboratoire de Math\'ematiques, 
Universit\'e de Rennes 1,
Campus de Beaulieu,
35042 Rennes Cedex, France}
\email{anna.lenzhen@univ-rennes1.fr}
  \author{Babak Modami}
  \address{ Department of Mathematics, Yale University, 10 Hillhouse Ave, New Haven, CT}
\email{babak.modami@yale.edu}
  \author   {Kasra Rafi}
  \address{Department of Mathematics, University of Toronto, Toronto, ON }
\email{rafi@math.toronto.edu}

  \date{\today}

\begin{abstract} 
  We construct an example of a  \Teich geodesic ray whose limit set in
  Thurston boundary of \Teich space is an $d$-dimensional simplex.
 
\end{abstract}
  
  \maketitle
  

\section{Introduction}

Thurston  introduced a compactification of \Teich space of a surface $S$ using 
a boundary space $\mathcal{PMF}(S)$ consisting of projective classes of measured foliations 
\cite{flp:TTs}. 
The boundary is homeomorphic to a sphere and the action of the mapping class group 
of the surface extends continuously to this boundary. In spite of the fact that \Teich 
metric is not negatively curved in any of the standard senses, using this compactification 
Thurston gave a classification of elements of mapping class groups in analogy with 
negatively curved spaces \cite{flp:TTs}.

In a hyperbolic space, every geodesic has a unique limit point. As a Teichm\"{u}ller counterpart, 
Masur \cite{masur:TB} showed that the limit set of a Teichm\"{u}ller geodesic ray
with a uniquely ergodic vertical foliation is a single point. However, Kerckhoff 
\cite{kerckhoff:AG} showed that Thurston boundary is not the visual boundary of the 
\Teich metric. 

In \cite{lenzhen:GL}, Lenzhen gave the first example where the limit set of a
\Teich geodesic ray is more than one point. The example is  for 
a surface of genus  two, and the limit set of the ray is an interval in 
one-dimensional simplex of measures for a non-minimal 
foliation in $\mathcal{PMF}(S)$. Since then, several other examples have been constructed. 
In \cite{rafi:LS},  it is shown that the same phenomenon can take place for 
a minimal foliation, with limit set being the entire one-dimensional simplex. 
In \cite{masur:LTG} an example of minimal foliation is constructed where the 
limit set of the corresponding ray is a proper subset of a one-dimensional simplex of 
measures and in \cite{rafi:LSII}  an example is constructed where the limit set is not simply 
connected and is homeomorphic to a circle.  Similar phenomena is also possible for the 
geodesic in \Teich space equipped with the Weil-Petersson metric
\cite{rafi:LWP, rafi:LWPN}. However, so far in all the examples the limit 
set has been at most one-dimensional. Masur has asked if the limit set can 
ever have higher dimension. In this paper, we give a positive answer to the 
question of Masur. 

\begin{theorem}\label{Thm:2D}
For any $d\geq 2$, there exists a \Teich geodesic ray whose limit set in $\mathcal{PMF}(S)$
is $d$-dimensional. 
\end{theorem}

 
The example is constructed 
as follows. Let $T^i,\, i \in \ZZ_{d+1} = \{0,1,2,\ldots,d\}$, be a square torus rotated so that the vertical 
direction has a slope $\theta^i\in (0,1)\setminus \QQ$ in $T^i$. Cut a vertical slit of size 
$s_0>0$ in $T_i$ and glue the left side of the slit of $T^i$ to the right side of $T^{i+1}$.  
We obtain a  translation surface, that is, a Riemann surface $X_0$ of genus $d+1$ 
(see \figref{surface}), with a holomorphic quadratic differential $(X_0,\phi_0)$ with two zeros of 
order $2d$ where the restriction of the vertical foliation to $T^i$ has slope $\theta^i$. 
Let ${\bf r}$ be the \Teich geodesic ray based at $X_0$, and in the direction of $\phi_0$. 
For each $i=0,1,\ldots, d$, let $\nu^i$ be the ergodic measured foliation in $\PMF(S)$ 
supported on $T^i$, and defined by  $\theta^i$.  
\begin{figure}[ht]
\setlength{\unitlength}{0.01\linewidth}
  \begin{picture}(100, 40)
  \put(23,-4){\includegraphics{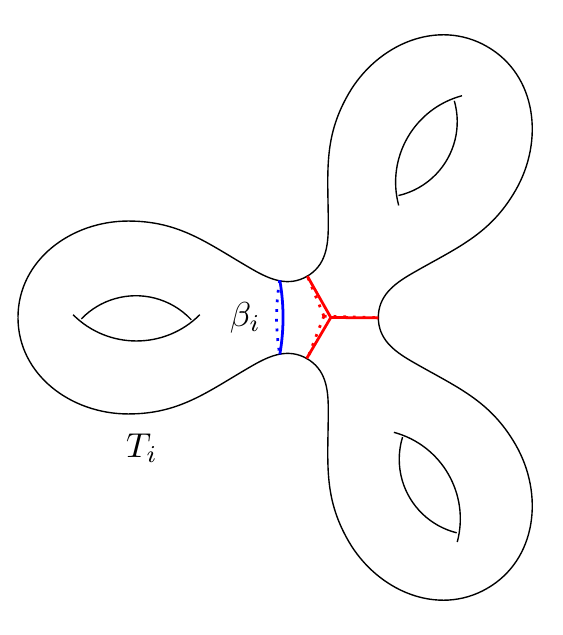}}
  \end{picture}        
  \caption{Case $d=2$. The surface $X_0$ is glued out of three tori.}
    
  \label{Fig:surface}
\end{figure} 
\thmref{2D} is a consequence of the following statement. 
\begin{theorem}\label{Thm:Ex}
There exist irrational numbers $\theta^0, \theta^1,\ldots,\theta^d$ such that the  limit set of 
the corresponding ray ${\bf r}$ is the simplex of measures spanned by $\nu^0, \nu^1,\ldots, \nu^d$.

\end{theorem}

The irrational numbers $\theta^i,\; i=0,\ldots, d,$  are defined via continued fraction expansions, where the coefficients of each continued fraction satisfies certain growth conditions (see $\S$\ref{subsec : cf setup}). The limit then is determined by estimating lengths of the curves corresponding to convergents of the continued fractions at different times along the ray {\bf r}.

\subsection*{Acknowledgement} The second author was partially supported by NSF grant DMS-1065872. The third author was partially supported by 
NSERC Discovery grant RGPIN-435885. The authors would like to thank the  referee for careful reading of the paper and many helpful comments.


\section{Background} 
\subsection*{Notation} 
 For a pair of sequences $\{x_n\}_{n\in\NN}$ and
$\{y_n\}_{n\in\NN}$, we write $x_n\sim y_n$ if
\[\frac{x_n}{y_n}\to 1  \quad\text{as}\quad n\to\infty.\]
Note that $\sim$ is an equivalence relation on sequences of numbers, in particular it is symmetric and transitive.

Let 
\[
P\RR_+^3=\Big\{[r,s,u] \ST r,s,u>0,\quad
(r,s,u)\equiv(\lambda r,\lambda s, \lambda u) \quad \forall \lambda>0 \Big\}.
\]
Similarly, for a pair of sequences $\{ [a_n^0, a_n^1, a_n^2]\}_n$ and
$\{ [b_n^0, b_n^1, b_n^2]\}_n$ in $P\RR^3$ we write 
\[
[a_n^0, a_n^1, a_n^2] \sim [b_n^0, b_n^1, b_n^2]
\qquad \text{if, for $i \in \ZZ_3$},\qquad
\frac{a_n^i}{a_n^{i+1}} \sim \frac{b_n^i}{b_n^{i+1}}.
\]

The notation $\emul$ means equal up to a multiplicative, $\eadd$ means 
equal up to an additive error and $\asymp$ means equal up to and additive and a multiplicative error with uniform constants. For example
\[
\Ga  \emul \Gb  \quad \Longleftrightarrow\quad
\frac{\Gb}K \leq \Ga \leq K \, \Gb, \qquad
\text{for a uniform constant $K$}. 
\]
The notations $\gmul$, $\gadd$ and $\succ$ are similarly defined. 
 
\subsection{Continued fractions}

Let $\theta=[a_0;a_1,a_2,\ldots]$ be any positive number, and denote the $n-$th convergent of $\theta$ by $\frac{p_n}{q_n}$. That is
\[
\frac{p_n}{q_n}=a_0+\cfrac{1}{a_1+\cfrac{1}{\dotsb +\cfrac{1}{a_n}}}. 
\]
We will need he following standard facts about continued fractions (see, for example, \cite{khinchin:CF}):
\begin{align} 
&q_n=a_nq_{n-1}+q_{n-2},\label{Eq:f1}\\
&\frac{1}{q_n+q_{n+1}}\leq|p_n-q_n\theta| \leq \frac{1}{q_{n+1}}.\label{Eq:f2}
\end{align}

\subsection{Teichm\"{u}ller theory}\label{subsec:Teichtheory}
 In this section we recall some background 
material mainly about Teichm\"{u}ller space and \Teich geodesics and throughout set our notations. 
We assume that the reader is familiar with basic facts about Teichm\"{u}ller space and the space of measured foliations.
See for example \cite{gardiner:QT} and \cite{flp:TTs} for a thorough treatment of this material. 

The {\it Teichm\"{u}ller space} of 
a closed orientable surface $S$, denoted by $\calT(S)$, is the space of equivalence classes of all marked Riemann surfaces
homeomorphic to $S$ i.e. orientation preserving homeomorphisms $f:S\to X$, where $X$ is a Riemann surface;  two marked surfaces 
$f_1:S\to X_1$ and $f_2:S\to X_2$ are equivalent 
 if $f_2\circ f_1^{-1}:X_1\to X_2$ is isotopic to a biholomorphic map. 
 
  A {\it measured foliation} $\nu$ on $S$ is a foliation with pronged singularities and a transverse measure.  
The space of measured foliations of $S$ is equipped with the weak$^*$ topology.  The projective class of a measured 
foliation $\nu$ 
is the class of all measures which are positive multiples of $\nu$. We denote the space of projective 
measured foliations by $\mathcal{PMF}(S)$ which 
 is equipped with the natural topology 
induced from the weak$^{*}$ topology of the space of measured foliations. 

  A quadratic differential $(X,\phi)$ on a Riemann 
surface $X$ is a $(2,0)-$tensor with holomorphic coefficients; in a local coordinate $z$ 
it has the form $\phi(z)dz^2$ with $\phi(z)$ a holomorphic function. 
Around every point where $\phi(z)$ is not zero, there exist 
 coordinates $\zeta=\xi+i\eta$, called {\it natural coordinates}, in which the quadratic differential
 can be represented as  $d\zeta^2$ (see e.g. \cite[\S 2]{gardiner:QT}). 
 There are two measured foliations naturally assigned to $\phi$.
 The trajectories $d\eta\equiv 0$
 and $d\xi\equiv 0$ define the horizontal and vertical foliations of $\phi$,
 respectively. Integrating $|d\eta|$ and $|d\xi|$ along arcs determine horizontal and vertical 
 measured foliations $\nu^+$ and $\nu^-$, respectively.
Moreover,  $(X,\phi)$ is defined uniquely by $X$ and its vertical measured foliation, by a theorem of
Hubbard and Masur \cite{masur:QDF}. 

A \Teich geodesic can be described as follows. Given 
a quadratic differential $(X_0,\phi_0)$ on $X_0$, let $\zeta=\xi+i\eta$ be a natural coordinate for $(X_0,\phi_0)$.
Then we can obtain a 1-parameter family $(X_t, \phi_t)$ of quadratic differentials 
defined locally by $d\zeta_t^2$ where $\zeta_t=e^t\xi+ie^{-t}\eta$. The map ${\bf g}:\RR\to \calT(S)$
which sends $t$ to $X_t$ is a \Teich geodesic. We will often write $(X_t,\phi_t)$ to refer to the geodesic.
We will also denote by ${\bf r}$ the \Teich ray which is the image of $\RR_+$.




\subsection*{Notions of length of a curve} 
By a \emph{curve} we mean the free homotopy class of an essential simple closed 
curve. There are various notions of length associated to a curve $\alpha$ on 
a surface with a quadratic differential $(X,\phi)$. 

We can equip $X$ with the hyperbolic metric in 
the conformal class of $X$ given by the uniformization. Then the {\it hyperbolic length} 
of $\alpha$, denoted by $\Hyp_X(\alpha)$, is the length of the geodesic representative 
of $\alpha$ on $X$. 

The {\it extremal length} of $\alpha$ is defined by 
\begin{equation}
\Ext_{X}(\alpha)=\sup_{\rho \in [X]}\frac{\ell_{\rho}(\alpha)^2}{\area_{\rho}(X)}
\end{equation}
where $\rho$ is any metric in the conformal class of $X$. The reciprocal of the extremal 
length is equal to the maximum modulus of any annulus with core curve $\alpha$ 
\cite{gardiner:QT}.

 Maskit \cite{maskit:HE}  established the following relation between hyperbolic and extremal lengths:
\begin{equation}\label{eq:Maskit}\frac{1}{\pi}\leq \frac{\Ext_X(\alpha)}{\Hyp_X(\alpha)}\leq \frac{1}{2}e^{\Hyp_X(\alpha)/2}.\end{equation}
When either $\Hyp_X(\alpha)$ or $\Ext_X(\alpha)$ is small, the above inequality implies that the two lengths
are comparable, 
\begin{equation}\label{eq:Maskit2}
\Hyp_{X}(\alpha)\emul \Ext_{X}(\alpha).
\end{equation}
 where the multiplicative constant depends only on an upper for the extremal or hyperbolic length of $\alpha^i$. 

The quadratic differential $(X,\phi)$  defines a singular flat metric $|\phi(z)||dz|^2$ on $X$. 
The {\it flat length} of $\alpha$, denoted by $\ell_\phi(\alpha)$, is the length of a 
geodesic representative of $\alpha$ in this metric.


Finally, let $\alpha'\sim\alpha$ be any curve in the homotopy class of $\alpha$. Recall then
 the notion of intersection number of a measured foliation $\nu$ and $\alpha$, defined by 
\[\I(\alpha,\nu):=\inf_{\alpha'\sim \alpha}\int_{\alpha'} \nu.\]
This generalizes the usual notion of geometric intersection number of two curves. 
\medskip

Given a quadratic differential $\phi$ with corresponding horizontal and vertical measured foliations $\nu^+$ and $\nu^-$, 
the horizontal length of the curve $\alpha$ is $h_{\phi}(\alpha)=\I(\alpha,\nu^-)$ and its vertical length is $v_\phi(\alpha)=\I(\alpha,\nu^+)$. 

Note that along a \Teich geodesic $(X_t,\phi_t)$ we have
\[h_{\phi_t}(\alpha)=e^th_{\phi}(\alpha)\;\;\mbox{and}\;\; v_{\phi_t}(\alpha)=e^{-t}v_{\phi}(\alpha).\] 

When the \Teich geodesic is fixed, to simplify our presentation  we will often use the notations $\Hyp_t(\alpha)$, $ \Ext_t(\alpha)$, 
$\ell_t(\alpha)$, $h_t(\alpha)$ and $v_t(\alpha)$ instead of writing
$\Hyp_{X_t}(\alpha)$, $ \Ext_{X_t}(\alpha)$, $\ell_{\phi_t}(\alpha)$, $h_{\phi_t}(\alpha)$ and $v_{\phi_t}(\alpha)$ respectively.

\subsection*{Balanced time}
 The balanced time of a curve $\alpha$ along a Teichm\"{u}ller geodesic $(X_t,\phi_t)$ 
is the time when the horizontal and vertical lengths of $\alpha$ are equal:
\[h_{\phi_t}(\alpha)=v_{\phi_t}(\alpha).\]
If the geodesic representative of $\alpha$ in the flat metric $|\phi||dz|^2$ is neither vertical i.e. $h_{\phi}(\alpha)\neq 0$ nor 
horizontal i.e. $v_{\phi}(\alpha)\neq 0$, then there is a unique balanced time for the curve $\alpha$ along the geodesic 
which we denote by $t_\alpha$.  The flat, extremal and hyperbolic lengths of $\alpha$ realize 
their minima in a uniformly bounded distance from the time $t_\alpha$. See \cite[\S 2]{minsky:CCI}\cite{rafi:HT} for more detail.  

 
\subsection*{Twist parameter}

 Let $X$ be a point in $\calT(S)$. For a curve $\alpha$ on $X$ let $Y_{\alpha}$ 
be the annular cover of $X$ associated to $\alpha$ i.e. the annular cover for which the curve $\alpha$ lifts to its core curve. 
Equip $Y_{\alpha}$ with the lift of the metric of $X$ and let $\overline{Y}_{\alpha}$ be the compactification of $Y_{\alpha}$ adding the 
ideal boundary. Let $\tau$ be an arc orthogonal to the core curve of $\overline{Y}_{\alpha}$ which connects the two boundaries 
of $\overline{Y}_{\alpha}$. Now the {\it twist parameter} of a curve $\gamma$ about the curve $\alpha$ is defined by
 \begin{equation}\twist_{\alpha}(\gamma,X):=\I(\tilde{\gamma},\tau)\end{equation}
 where $\tilde{\gamma}$ is any chosen lift of $\gamma$ that intersects the core of $\overline{Y}_{\alpha}$.

For a \Teich geodesic $(X_t, \phi_t)$ Rafi \cite[Theorem 1.3]{rafi:CM} gives the 
following estimate for the twist parameter of a curve $\gamma$ about $\alpha$ at time $t$:
\begin{eqnarray}\label{eq : twist}
 \left|\twist_{\alpha}(\gamma,X_t)\right|&\leq&  \frac{c_\gamma}{\Hyp_{X_t}(\alpha)} \;\;\;\mbox{ if } t\leq t_\alpha,\\
|\twist_{\alpha}(\gamma,X_t)- \I_{\alpha}(\nu^-,\nu^+)|&\leq& \frac{c_\gamma}{\Hyp_{X_t}(\alpha)} \;\;\;\mbox{ if } t> t_{\alpha}.\nonumber
\end{eqnarray}
Here, the number $\I_{\alpha}(\nu^-,\nu^+)$ is the maximum number that a leaf of the lift 
of $\nu^-$ and a leaf of the lift of $\nu^+$ to $Y_{\alpha}$ intersect. The number, up to 
an additive constant, is equal to the maximum of the number of times that a leaf of $\nu^-$ 
and a leaf of $\nu^+$ intersect inside of the maximal flat cylinder with core curve 
$\alpha$  (see below for more detail about the maximal flat cylinder). Finally, note that the constant $c_\gamma$ depends on $\gamma$.

 \begin{remark}
In fact, Rafi states the estimate for $\twist_\alpha(\nu^+,X_t)$, which using the fact that the intersection number $\I_\alpha(\cdot,\cdot)$ is quasi-additive
 and absorbing $\I_\alpha(\gamma,\nu^+)$ in the $O$ notation constant gives us the above estimate.
\end{remark}
 
 In what follows we recall estimates for the extremal and hyperbolic lengths of a curve at a point in the \Teich space which 
 we will use later in the paper. 
 
\subsection*{An estimate for the extremal length of a curve}
Using the flat structure of $(X,\phi)$, one can estimate the  extremal length of 
a curve $\alpha$ on $X$. In general $\alpha$ does not have a unique geodesic representative 
with respect to the flat metric of $\phi$. However, the set of geodesic representatives foliate a
(possibly degenerate) flat cylinder $F_\alpha$ in $(X,\phi)$. Let $f_\alpha$ be
the distance between the boundaries of $F_\alpha$. Then 
$\Mod_X(F_\alpha) = \frac {f_\alpha}{\ell_\phi(\alpha)}$ where $\Mod(\param)$ is 
the modulus of the annulus. For either boundary component of $F_\alpha$, we consider 
the largest one-sided regular neighborhood of $F_\alpha$ that is an embedded annulus. 
We denote these annuli by $E_\alpha$ and $G_\alpha$ respectively and refer to them 
as expanding annuli associated to $\alpha$. Denote the distance between boundaries 
of $E_\alpha$ and $G_\alpha$ (i.e., the radius of the associated regular neighborhood) 
by $e_\alpha$ and $g_\alpha$ respectively. When $e_\alpha > \ell_\phi(\alpha)$, we 
 have
\[
\Mod_X(E_\alpha) \emul \log \frac{e_\alpha}{\ell_\phi(\alpha)}. 
\]
The same holds for $g_\alpha$ and $G_\alpha$. We can then estimate
the extremal length of $\alpha$ as follows (see \cite{minsky:HM, rafi:HT}):
\begin{equation}\label{eq : ext length modulus of annuli}\frac{1}{\Ext_X(\alpha)}\emul\Mod_{X}(E_{\alpha})+\Mod_{X}(F_{\alpha})+\Mod_{X}(G_{\alpha}).\end{equation}

\subsection*{An estimate for the hyperbolic length of a curve}

Given $L>0$, let $P$ be a pants decomposition of $X$, i.e. a maximal collection of pairwise disjoint
closed curves,  with the property that the hyperbolic 
lengths of all curves in $P$ are at most $L$. 

For a curve $\alpha$ on $X$ let the {\it width of $\alpha$}, $\width_X(\alpha)$, be the width of the collar around $\alpha$ from 
the Collar lemma \cite[\S 4.1]{buser:GSC}. We have the following estimate for the width
\begin{equation}
 \label{Eq:width}\width_X(\alpha)=2\arcsinh{\frac{1}{\sinh(\frac 1 2 \Hyp_X(\alpha))}}\eadd -2\log(\Hyp_{X}(\alpha)).
\end{equation}
Now define the contribution to the length of a curve $\gamma$ from a curve $\alpha\in P$ by
\begin{equation}\label{eq : contribution}\Hyp_X(\gamma,\alpha)=\I(\gamma,\alpha)\big(\width_X(\alpha)+
\twist_X(\gamma,\alpha)\Hyp_X(\alpha)\big)
\end{equation}
where the additive constant depends only on the topological type of the surfaces.
Then we have the following estimate for the hyperbolic length of a curve $\gamma$ in terms of the contributions from the curves in $P$, 
\begin{equation}\label{eq : length contribution}
\big|\Hyp_X(\gamma)-\sum_{\alpha\in P}\Hyp_{X}(\gamma,\alpha)\big|=O\big(\sum_{\alpha\in P}\I(\gamma,\alpha)\big).
\end{equation}
where the constant of the $O$ notation depends only on $L$. See \cite[Lemma 3.7]{rafi:LT}.

\subsection*{Growth of hyperbolic length along a Teichm\"{u}ller geodesic}
 It follows from Wolpert's estimate for the change of length
 \cite[Lemma 3.1]{wolpert:LS} and the description of Teichm\"{u}ller geodesics that the hyperbolic length of a curve varies at 
 most exponentially along a Teichm\"{u}ller geodesic. More precisely, given times $t,s\in\mathbb{R}$ with $t\geq s$ we have
\begin{equation}\label{eq : wolpert}e^{-2|t-s|}\Hyp_s(\alpha)\leq\Hyp_t(\alpha)\leq e^{2|t-s|}\Hyp_s(\alpha)\end{equation}
The above inequality and the \eqnref{width} in particular show that the width of the collar of the curve $\alpha$ grows at most linearly along a \Teich geodesic. 

 \subsection*{Thurston boundary} 
  The main purpose of this paper is  the construction of \Teich geodesic rays with two-dimensional 
 limit sets in the Thurston boundary. The Thurston boundary of the \Teich space $\calT(S)$ is the space of projective measured
  foliations on $S$, $\mathcal{PMF}(S)$. A sequence of points $X_n\in \calT(S)$ converges to the projective class of a measured
   foliation $[\nu]$ if and only if for any two curves $\gamma_1,\gamma_2$ on $S$ we have
 \[
 \lim_{n\to\infty}\frac{\Hyp_{X_n}(\gamma_1)}{\Hyp_{X_n}(\gamma_2)}=\frac{\I(\gamma_1,\nu)}{\I(\gamma_2,\nu)}.
 \]
The topology defined by this notion of convergence turns $\calT(S)\cup \mathcal{PMF}(S)$ into a closed ball where 
$\mathcal{PMF}(S)$ is the boundary sphere. For more detail see  \cite[expos\'e 8]{flp:TTs}.

\section{The Teichm\"{u}ller geodesic ray and its limit set}

In this section we prove our main result.
 First, in $\S$\ref{subsec : cf setup}, via continued fraction expansions, we define a measured foliation on $X_0$ and hence 
 fix a \Teich ray based at $X_0$. Then, in $\S$\ref{subsec : time length}, we find the shortest pants decomposition at various 
 times along the geodesic ray, to then be able to estimate hyperbolic length of curves using \eqref{eq : length contribution}. 
 In $\S$\ref{subsec : limit sets} we use this information to determine the limit set of the \Teich ray and prove \thmref{Ex}. 
 To keep the exposition fairly simple,  the proof given here is for $d=2$. For $d>2$ the notation is significantly heavier while the arguments
 are exactly the same. 

\subsection{Setup of continued fraction expansions}\label{subsec : cf setup}

Let $\{[u^0_k,u^1_k,u^2_k]\}_{k\in\NN}$ be a dense sequence in $P\RR_+^3$ where 
$u^0_k, u^1_k,u^2_k\in \NN$. Given this sequence, we will choose the numbers $\theta^i$ 
by describing their continued fraction expansion coefficients. 

Let $\{a_j^0\}_{j\in \NN}, \{a_j^1\}_{j\in \NN}$ and $\{a_j^2\}_{j\in \NN}$ be three sequences 
of positive  integers defined inductively as follows. Set $a_1^0=a_1^1=a_1^2=1$. 
Now, for $k\geq 1$ and $i \in \ZZ_3$, assume $a^i_1, \ldots, a^i_{2k-1}$ are defined and
whenever $n$ is such that $a^i_1,\ldots,a^i_n$ are defined, let 
\[
\frac{p_n^i}{q_n^i}=\cfrac{1}{a_1^i+\cfrac{1}{\dotsb +\cfrac{1}{a_n^i}}}.
\]
Choose $a_{2k}^i, a_{2k+1}^i \in \NN$ so that 
\begin{enumerate}[(i)]
\item \label{eq : q1} $a_{2k}^i > k\cdot \max \big\{a_{2k-1}^i,u^0_k, u^1_k, u^2_k\big\}$,
\item \label{eq : q2} $[a_{2k}^0, a_{2k}^1, a_{2k}^2] = [u^0_k, u^1_k, u^2_k]$ as elements in $P\RR_+^3$,
\item \label{eq : q3} $a_{2k+1}^i > \exp(ka^i_{2k})$, 
\item \label{eq : q4}$a_{2k+1}^0q_{2k}^0=a_{2k+1}^1q_{2k}^1=a_{2k+1}^2q_{2k}^2$.
\end{enumerate}
Define $\theta^i =  [0;a_1^i,a_2^i,\ldots] \in (0,1)$. That is
\[
\theta^i=\cfrac{1}{a_1^i+\cfrac{1}{a_2^i +\cfrac{1}{a_3^i+\dotsb}}} .
\]


\begin{lemma} \label{Lem:slopes}Let $\theta^i, i\in \ZZ_3$ be as above. Then 
$\theta^i$ are irrational and, for every $k$ , we have 
\begin{align}
&q_{2k+1}^0=q_{2k+1}^1=q_{2k+1}^2,\label{eq : q_2k+1}\\
&[q_{2k}^0, q_{2k}^1, q_{2k}^2] \sim [u^0_k, u^1_k, u^2_k], \label{eq : q_2k}\\ 
&\frac{\log {a_{2k+1}^i}}{ \max \{a_{2k}^i,u^0_k,u^1_k,u^2_k\}}\to \infty \;\text{as}\; k\to\infty, \label{eq : log a_2k+1/ log max}\\
&\log {a_{2k+1}^0}\sim \log {a_{2k+1}^1}\sim \log {a_{2k+1}^2},\label{Eq:aequiv}\\
&q_n^i\emul \prod_{j=1}^{n}a_j^i.\label{Eq:qprod}
\end{align}
\end{lemma}

\begin{proof} The irrationality of $\theta^i$ follows from the fact  that the coefficients 
$a_n^i$ are non-zero. 

We prove (\ref{eq : q_2k+1}) by induction on $k$. By setup of the continued fraction expansion, we have $a^0_1=a^1_1=a^2_1$, and therefore $q^0_1=q^1_1=q^2_1$ by (\ref{Eq:f1}). 
Now assume that (\ref{eq : q_2k+1}) holds for all $k'$ less than or equal to some $k>1$. 
For each $i \in \ZZ_3$ we have that $q^i_{2k}=a^i_{2k}q^i_{2k-1}+q^i_{2k-2}$ by (\ref{Eq:f1}). 
Moreover, by (\ref{eq : q4}), for $i,j\in\ZZ_3$, we have $a^i_{2k+1}q^i_{2k}=a^j_{2k+1}q^j_{2k}$. 
These two equalities and assumption of the induction imply that (\ref{eq : q_2k+1}) holds for $k$ as well. 

To see \eqref{eq : q_2k}, note that $a^0_{2k}=\frac{q^0_{2k}-q^0_{2k-2}}{q^0_{2k-1}}$ and $a^1_{2k}=\frac{q^1_{2k}-q^1_{2k-2}}{q^1_{2k-1}}$ by (\ref{Eq:f1}). 
Dividing these two numbers and taking into account that $q^0_{2k-1}=q^1_{2k-1}$ by (\ref{eq : q_2k+1}) we get 
\[\frac{a^1_{2k}}{a^0_{2k}}=\frac{q^1_{2k}-q^1_{2k-2}}{q^0_{2k}-q^0_{2k-2}}=\frac{q^1_{2k}}{q^0_{2k}}\frac{1-(q^1_{2k-2}/q^1_{2k})}{1-(q^0_{2k-2}/q^0_{2k})}.\]
The growth of the sequence $\{q^i_k\}_k$ from (\ref{eq : q1}) and (\ref{eq : q3}) implies that $q^1_{2k-2}/q^1_{2k}$ and $q^0_{2k-2}/q^0_{2k}$ go to $0$ as $k\to\infty$. 
Therefore, $a^0_{2k}/a^1_{2k}\sim q^0_{2k}/q^1_{2k}$. But, by (\ref{eq : q2}), $a^0_{2k}/a^1_{2k}=u^0_k/u^1_k $. Hence
\[
q^0_{2k}/q^1_{2k} \sim u^0_k/u^1_k. 
\]
Similarly, we can show that 
\[
q^1_{2k}/q^2_{2k} \sim u^1_k/u^2_k\;\;\text{and}\;\; q^2_{2k}/q^0_{2k} \sim u^2_k/u^0_k .
\]
This finishes the proof of Equation \eqref{eq : q_2k}.
\medskip

To see (\ref{eq : log a_2k+1/ log max}), note that by (\ref{eq : q3}) we have 
\[
\frac{\log a^i_{2k+1}}{\max\{a^i_{2k},u^0_k,u^1_k,u^2_k\}}>\frac{ka^i_{2k}}{\max\{a^i_{2k},u^0_k,u^1_k,u^2_k\}}.
\]
The term $\max\{a^i_{2k},u^0_k,u^1_k,u^2_k\}$ is either $a^i_{2k}$, or $\max\{ u^0_k,u^1_k,u^2_k\}$. In the first situation, the right-hand side of the above inequality is equal to $k$. 
In the second situation, note that by (\ref{eq : q1}), $\frac{a^i_{2k}}{\max\{u^0_k,u^1_k,u^2_k\}}>k$. Hence, the right-hand side of the above inequality is at least $k^2$.
Thus in both cases the right-hand side goes to $\infty$ as $k\to\infty$, and therefore (\ref{eq : log a_2k+1/ log max}) holds.
\medskip
 
Let us now prove (\ref{Eq:aequiv}). Without loss of generality suppose that $i=0$ and $j=1$. By (iv) we have 
\begin{eqnarray*}
\log a_{2k+1}^0/\log a_{2k+1}^1=1+\frac{\log(q_{2k}^1/q_{2k}^0)}{\log a_{2k+1}^1}.
\end{eqnarray*}
Moreover, $\log(q_{2k}^1/q_{2k}^0)\leq q_{2k}^1/q_{2k}^0$ and by (\ref{eq : q_2k}) $q_{2k}^1/q_{2k}^0\sim u^1_k/u^0_k$. 
Then the right-hand side above goes to $1$, because by (\ref{eq : log a_2k+1/ log max}), $\frac{u^1_k}{u^0_k\log a_{2k+1}^1}\to 0$ as $k\to\infty$.
 This finishes the proof of (\ref{Eq:aequiv}).
\medskip
 
We are left to prove \eqnref{qprod}.
It follows  from \eqnref{f1} that 
\[
q^i_n\geq \prod_{j=1}^{n}a_j^i,
\] 
which is the lower bound in (\ref{Eq:qprod}). For the upper bound, we write
\[\frac{q_n^i}{ \prod_{j=1}^{n}a_j^i}=\frac{a_n^iq_{n-1}^i+q_{n-2}^i}{\prod_{j=1}^{n}a_j^i}\leq \frac{q_{n-1}^i}{\prod_{j=1}^{n-1}a_j^i}\left(1+\frac{1}{a_{n}^i}\right),\]
and so by an induction on $n$ we get
\[\frac{q_n^i}{ \prod_{j=1}^{n}a_j^i}\leq \prod_{j= 1}^n\left(1+ \frac{1}{a_j^i}\right)< \prod_{n \geq 1}\left(1+ \frac{1}{a_n^i}\right).\]
The infinite product on the right-hand side converges and is uniformly bounded for any sequence of 
coefficients $\{a_n^i\}_n$ that satisfies conditions (i) and (iii).
 \end{proof}
 
 For the rest of the paper let $(X_0,\phi_0)$ be the Riemann surface and quadratic differential obtained by gluing the three rotated square tori $T^i$ along
 vertical slits. The foliations $\nu^i$ on $T^i$ in the directions with slopes $\theta^i$, $i\in \ZZ_3$, glue together to make the vertical foliation $\nu$ of $\phi_0$.

 \subsection{Time and length estimates}\label{subsec : time length}

Denote by  $\alpha_n^i$ the simple curve on $T^i$ with slope $\frac{p_n^i}{q_n^i}$. 
Then we have the following.

\begin{lemma}\label{Lem:conv}
For $i=0,1,2$ the curves $\alpha_n^i$ converge to $\nu^i$ in $\mathcal{PMF}(X_0)$. More precisely, for any  simple closed curve $\gamma$ on $X_0$,
\begin{equation}\label{Eq:coeffs}
\frac{1}{q_n^i}\I(\gamma,\alpha_n^i)\to \I(\gamma,\nu^i)\sqrt{1+(\theta^i)^2}.
\end{equation}
\end{lemma}
\begin{proof}
Suppose first that $\gamma$ is a curve  on one of the tori, say $T^i$, and suppose that it has 
 slope $\frac{p}{q}$. Then 
we have $\I(\gamma,\alpha_n^i)=|pq_n^i-qp_n^i|$. On the other hand, 
the intersection number of $\gamma$ with $\nu^i$ is the absolute value of the dot product of the vector $(p,q)$
and the vector of unit length perpendicular to the foliation $\nu^i$, namely
$\frac{1}{\sqrt{1+(\theta^i)^2}}(\theta^i, -1)$.  Hence,
\[
\I(\gamma, \nu^i) = \frac{|p-q\theta^i|}{\sqrt{1+(\theta^i)^2}}.
\]
Since $\frac{p_n^i}{q_n^i}\to \theta^i$, we have
\[
\frac{1}{q_n^i}\I(\gamma,\alpha_n^i) =
  \left|p-q\frac{p_n^i}{q_n^i}\right|\to |p-q\theta^i|=\I(\gamma,\nu^i)\sqrt{1+(\theta^i)^2},
\]
as $n\to\infty$, which is \eqnref{coeffs} for $\gamma\subset T^i$. Now, since a measured foliation on $T^i$ is uniquely determined by
its intersection number with all simple closed curves on $T^i$, it follows that
\[
\frac{1}{q_n^i\sqrt{1+(\theta^i)^2}}\alpha_n^i\to\nu^i,\;\text{as $i\to\infty$}
\qquad\text{in $\mathcal{MF}(T^i)$}.
\] 
Furthermore, since $\mathcal{MF}(T^i)$ embeds continuously into 
$\mathcal{MF}(X_0)$, the sequence converges in $\mathcal{MF}(X_0)$ as well. 
Thus, \eqnref{coeffs} holds for all curves on $X_0$, which
 finishes proof of the lemma.
\end{proof}

From (\ref{eq : q_2k}) in \lemref{slopes} and  \lemref{conv} we have 

\begin{corollary}\label{Cor:ratios}
 For any curve $\gamma$ on $X_0$, we have the following equivalence:
 \begin{equation*}
\sum_{i \in \ZZ_3} \I(\gamma,\alpha_{2k}^i)
  \sim \frac{q_{2k}^0}{u^0_k} 
  \left(\sum_{i \in \ZZ_3} u^i_k \sqrt{1+(\theta^i)^2}\I(\gamma,\nu^i) \right).
 \end{equation*}
\end{corollary}

To estimate hyperbolic length  of some fixed curve at a certain time along ${\bf r}$ using  (\ref{eq : length contribution}) we need information about the short curves at that time. 
The next lemma shows that  the curves $\alpha_n^i$ become short along the ray ${\bf r}$ and gives estimates for the shortest lengths of the curves and the twist about them,
 also the times when the curves are shortest.

\begin{lemma}\label{Lem:shortcurves} For $i\in\ZZ_3$ and $n\in \NN$, the curve $\alpha_n^i$ is balanced at time
\begin{equation}\label{Eq:balanced}
 t_n^i\eadd \frac{1}{2}\log {q_n^iq_{n+1}^i}\eadd \sum_{j=1}^n\log a_j^i+\frac{1}{2}\log a_{n+1}^i.
 \end{equation}
The flat length of  $\alpha_n^i$ is minimal at $t_n^i$ and is given by
\begin{equation}\label{Eq:flat}
    \ell_{t_n^i}(\alpha_n^i)\emul\frac{1}{\sqrt{a_{n+1}^i}}.
\end{equation}
 Moreover, the extremal and hyperbolic lengths of $\alpha_n^i$ 
 at $t_n^i$ are comparable and
\begin{equation}\label{Eq:Hyp}
  \Ext_{t^i_n}(\alpha^i_n)\emul \Hyp_{t^i_n}(\alpha_n^i)\emul \frac{1}{a_{n+1}^i}.
\end{equation}
Finally, we have
\begin{equation}\label{Eq:tw}
\I_{\alpha^i_n}(\nu^-,\nu^+)\emul a^i_{n+1}
\end{equation}

\end{lemma}
\begin{proof}
The time when $\alpha_n^i$ is balanced can be computed explicitly. We have 
\begin{equation}\label{Eq:formgen}
\ell_{t}(\alpha^i_n)^2=h_t^2(\alpha^i_n)+v_t^2(\alpha^i_n).
\end{equation} 
Since
$h_t(\alpha^i_n)=e^t\frac{1}{\sqrt{1+(\theta^i)^2}}(p^i_n-\theta^iq^i_n)$ and $v_t(\alpha^i_n)=e^{-t}\frac{1}{\sqrt{1+(\theta^i)^2}}(q^i_n+\theta^ip^i_n)$, we have 
\begin{equation}\label{Eq:formalpha}
\ell_{t}(\alpha^i_n)^2=\frac{1}{(1+(\theta^i)^2)}\Big(e^{-2t}(q^i_n+\theta^i p^i_n)^2+e^{2t}(p^i_n-\theta^i q^i_n)^2\Big).
\end{equation} 
Now a straightforward calculation shows that $\ell_t(\alpha^i_n)^2$ reaches its minimum at the time 
\begin{equation}\label{eq : tin pnqn}
t_n^i=\frac{1}{2}\log\frac{p_n^i\theta^i+q_n^i}{|q_n^i\theta_n^i-p_n^i|}
\end{equation}
(for more details see \cite[Lemma 1]{lenzhen:GL}).
\begin{remark}
In \cite{lenzhen:GL} Lenzhen uses the parametrization $\zeta_t=e^{t/2}\xi+ie^{-t/2}\eta$ for the Teichm\"{u}ller geodesic, 
where $\zeta=\xi+i\eta$ is a natural coordinate at time $0$. But in this paper we use the parametrization 
$\zeta_t=e^{t}\xi+ie^{-t}\eta$ of the geodesic, which introduces the extra $\frac{1}{2}$ in the above formula.
\end{remark}
By \eqnref{f2} we have   
 \begin{equation}\label{eq : ineq pnqn}
 q^i_{n+1}(q^i_n\theta^i+p^i_n)\leq \frac{p_n^i\theta^i+q_n^i}{|q_n^i\theta_n^i-p_n^i|}\leq (q^i_n+q^i_{n+1})(q^i_n\theta^i+p^i_n).
 \end{equation}
Now note that, $\lim_{n\to\infty}\frac{p^i_n}{q^i_n}=\theta^i$ and hence $q^i_n\emul p^i_n$ for all  sufficiently large $n$. Moreover,
 $\theta^i \in(0,1)$, in fact since $a^i_1=1$, we have that $\theta^i>\frac{1}{2}$,  so the multiplicative constant in the coarse equality
 $q^i_n\emul p^i_n$ is independent of $\theta^i$. Then by (\ref{eq : ineq pnqn}) and since $q^i_{n+1}\geq q^i_n$ we have 
\begin{equation}\label{eq : ptheta/qtheta}\frac{p_n^i\theta^i+q_n^i}{|q_n^i\theta_n^i-p_n^i|}\emul q^i_nq^i_{n+1}.\end{equation}
 The equations (\ref{eq : ptheta/qtheta}) and  (\ref{eq : tin pnqn}) give us 
\[ t_n^i\eadd \frac{1}{2}\log q_n^iq_{n+1}^i.\]
 The rest of \eqnref{balanced} now follows from \eqnref{qprod} of \lemref{slopes}. 
 
 Moreover, since the times $t^i_n\to\infty$, it follows from \cite[Lemma 3]{lenzhen:GL} and its proof, which essentially
  uses the fact that the area of the maximal flat cylinder  $\cyl_t(\alpha_n^i)$ with core curve $\alpha_n^i$ tends to 1, that 
 \begin{equation}\label{Eq:allsame}
 \ell_{t^i_n}(\alpha^i_n)^2\sim\Ext_{t^i_n}(\alpha^i_n)\sim\frac{ 1} {\Mod(\cyl_{t^i_n}(\alpha^i_n)) } . 
 \end{equation}

  Furthermore, the fact that by definition the sequence $a^i_n$ goes to infinity and (\ref{Eq:f2}) imply that 
 \[\ell_{t^i_n}(\alpha^i_n)^2\sim \frac{2q^i_n}{q^i_{n+1}}\sim \frac{2}{a^i_{n+1}},\]
(for more detail see the proof of \cite[Corollary 1]{lenzhen:GL}). This gives us (\ref{Eq:flat}).

 Then \eqnref{Hyp} follows from the comparison of extremal and hyperbolic lengths (\ref{eq:Maskit2}).  
 
Finally, by  \cite[Proposition 5.8]{rafi:LT}, we have
\[\I_{\alpha^i_{n}}(\nu^-,\nu^+)\emul\Mod(\cyl_{t^i_n}(\alpha^i_n))\] 
as long as 
$\Ext_{t^i_n}(\alpha^i_n)$, up to a bounded multiplicative constant, is $\frac{1}{\Mod(\cyl_{t^i_n}(\alpha^i_n))}$, which is 
the case by \eqnref{allsame}.  \eqnref{tw} now follows from \eqnref{Hyp}. This finishes the proof of the lemma.
\end{proof}

\begin{figure}
    \begin{center}
    \centerline{\includegraphics{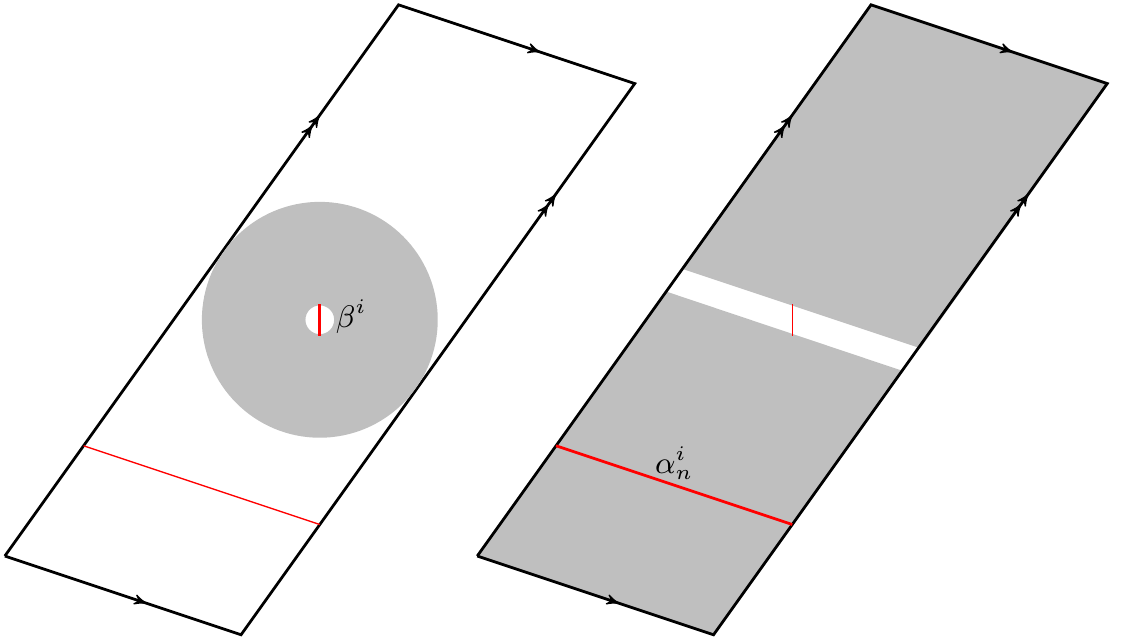}}
    \end{center}
    \caption{Annuli in $T^i$ about $\beta^i$ and $\alpha_n^i$. Expanding annulus with core curve $\beta^i$, 
    on the left, and flat annulus about $\alpha_n^i$, on the right, at the time when $\alpha_n^i$ is balanced.}
    
    \label{Fig:annuli}
    \end{figure} 

There are three other curves, namely $\beta^i=\partial T^i$, that become very short 
along ${\bf r}$. In fact, the length of $\beta^i$  goes to $0$. We have the following estimate for the length of $\beta^i$:

\begin{lemma}\label{Lem:dividing}
For $i\in\ZZ_3$ we have
\begin{equation}
\Hyp_{t_n^i}(\beta^i)\emul \frac{1}{\log{q_n^i}}
 \end{equation}
\end{lemma}

\begin{proof}
Starting with the proof, note that
 since the curve $\beta^i$ is homotopic to the union of two critical trajectories of the quadratic differential $\phi$ connecting 
 two critical points of $\phi$ (see Figure \ref{Fig:surface}), the flat length of $\beta^i$ is 
 \[\ell_{t_n^i}(\beta^i)=2s_0e^{-t_n^i},\]
  where $s_0$ is the size (flat length) of the slit we cut on the tori $T^i$, $i\in \ZZ_3$, to produce the initial genus three flat surface. 
  Moreover, the shortest curve on $T^i$ at $t_n^i$ is $\alpha_{n}^i$, which is balanced and whose flat length satisfies 
  \[\ell_{t_n^i}(\alpha_{n}^i)\emul \frac{1}{\sqrt{a_{n+1}^i}}\]
   by \lemref{shortcurves}. Now by the  two estimates above 
\begin{eqnarray}\label{eq : length ratio}
\log\frac{\ell_{t^i_{n}}(\alpha^i_{n})}{\ell_{t^i_n}(\beta^i)}&\eadd& t^i_n-\frac{1}{2}\log a^i_{n+1}\\
&\eadd&\sum_{j=1}^n \log a^i_j \eadd \log q^i_n\nonumber
\end{eqnarray}
where the second equality holds by (\ref{Eq:balanced}) in \lemref{shortcurves} and the third equality by (\ref{Eq:qprod}) in \lemref{slopes}. 
  
 Further, note that there is no flat annulus around $\beta^i$, and the distance between boundaries of the largest embedded neighborhood of $\beta^i$ inside $T^i$ is  
 \[\frac{\ell_{t^i_n}(\alpha^i_{n})-\ell_{t^i_n}(\beta^i)}{2}\]
  (see the right-hand side of Figure \ref{Fig:annuli}). Hence by Equation (\ref{eq : ext length modulus of annuli}) we have 
\begin{equation}\label{eq : ext tni betai}
\frac{1}{\Ext_{t_n^i}(\beta^i)}
\emul \log \left(\frac{\ell_{t^i_n}(\alpha^i_{n})-\ell_{t^i_n}(\beta^i)}{2\ell_{t^i_{n}}(\beta^i)}\right)
 =\log\left(\frac{\ell_{t^i_n}(\alpha^i_{n})}{\ell_{t^i_{n}}(\beta^i)}-1\right)-\log 2.
\end{equation}
Also since $q^i_n\to \infty$ as $n\to\infty$, by (\ref{eq : length ratio}) we have $\frac{\ell_{t^i_{n}}(\alpha^i_{n})}{\ell_{t^i_n}(\beta^i)}\to \infty$ as $n\to\infty$. 
Thus from (\ref{eq : ext tni betai}) we may deduce that 
\[\frac{1}{\Ext_{t_n^i}(\beta^i)}\emul\log\frac{\ell_{t^i_{n}}(\alpha^i_{n})}{\ell_{t^i_n}(\beta^i)}.\] 
Then appealing again to (\ref{eq : length ratio}) we have that the extremal length of $\beta^i$ at $t_n^i$ satisfies
\[\Ext_{t_n^i}(\beta^i)\emul \frac{1}{\log{q_n^i}}.\]
The lemma now follows from Maskit's comparison of hyperbolic and extremal lengths 
\eqref{eq:Maskit2}. 
\end{proof}
The following lemma follows from the proof of  \cite[Theorem 1.2]{rafi:CM}).

\begin{lemma}\label{Lem:Kasra}
For any $i\in \ZZ_3$  and any $t>s$, the hyperbolic length of $\beta^i$ satisfies 
\begin{equation}
\frac{1}{\Hyp_{s}(\beta^i)}\succ \frac{1}{\Hyp_t(\beta^i)}.
\end{equation} 
\end{lemma}
This and \lemref{dividing} imply

\begin{corollary}\label{Cor:beta_zero}
For $i\in\ZZ_3$ and  for $t\in[t_n^i,t_{n+1}^i]$ we have
\begin{equation}\label{beta_decreasing}
\Hyp_{t_{n+1}^i}(\beta^i)\lmul\Hyp_{t}(\beta^i)\lmul\Hyp_{t_n^i}(\beta^i).
\end{equation}
In particular,
\begin{equation}\label{beta_zero}
\underset{t\to\infty}{\lim}\Hyp_{t}(\beta^i)= 0.
 \end{equation}
\end{corollary}
\begin{proof}
Let $t\geq t_n^i$. By \lemref{Kasra} there is $K\geq 1$ independent of $t$ and $t_n^i$ such that
$$
\frac{1}{\Hyp_{t}(\beta^i)}\geq \frac{1}{K\Hyp_{t_n^i}(\beta^i)}-K.
$$

From  \lemref{dividing} and the fact that $q^i_{n}\to\infty$ we see that the
expression on the right is positive for $n$ big enough, and hence
$$\Hyp_{t}(\beta^i)\lmul \Hyp_{t_n^i}(\beta^i).$$ 
The other inequality can be shown in a similar way. Now, since along $\{t_n^i\}_n$ the hyperbolic length
of $\beta^i$ goes to 0, we are done. 
\end{proof}

\subsection{The limit set} \label{subsec : limit sets}

To find the limit set of the geodesic ray ${\bf r}$, 
we examine the geometry of Riemann surface $X_{t_n}$
for a carefully chosen sequence of times $\{t_n\}_n$.
%
The curves $\beta^i$ are always short and the curves $\alpha^i_{2n}$ get short 
roughly at the same time $t_n$. The hyperbolic length of any given curve $\gamma$
can be computed as the sum of the contributions to the length of $\gamma$
coming from crossing the short curves in $X_{t_n}$. We will see that the contribution
from  $\alpha^i_{2n}$ dominates the contribution from $\beta^i$. 
But the curves $\alpha^i_{2n}$ are chosen so that the length contributions
coming from these curves, thought of as a projective triple, form a dense subset of 
$P\RR^3$. This will let us conclude that the limit set of the ray ${\bf r}$ contains the whole simplex of 
projective measures. The fact that the limit is contained in the simplex follows from a similar argument 
showing that asymptotically along the ray the contribution of $\beta^i$ to the length of $\gamma$
is negligible.

\begin{proof}[Proof of \thmref{Ex}] 

We first show that the limit set of ${\bf r}$  contains the simplex spanned by projective classes of the measures
$\nu^0, \nu^1$ and $\nu^2$.  
 For this purpose we show that there exists a
sequence of times $t_n\to \infty$ such that, given any two curves $\gamma_1$ and 
$\gamma_2$ not equal to $\beta^i, i\in \ZZ_3$, we have
 \begin{equation}\label{Eq:ratio}
 \frac{\Hyp_{t_n}(\gamma_1)}{\Hyp_{t_n}(\gamma_2)}\sim 
 \frac{\sum_{i \in \ZZ_3}w^i_n\I(\gamma_1,\nu^i)}
        {\sum_{i \in \ZZ_3}w^i_n\I(\gamma_2,\nu^i)}
\end{equation}
where $w^i_n=u^i_n\sqrt{1+(\theta^i)^2}$, $i\in\ZZ_3$. But the set $\{[u^0_n,u^1_n,u^2_n]\}_{n\in \NN}$ is dense in $P\RR_+^3$ and the map
\[
[a,b,c] \to \Big[a\sqrt{1+(\theta^0)^2},b\sqrt{1+(\theta^1)^2},c\sqrt{1+(\theta^2)^2} \Big] 
\] 
is a homeomorphism of $P\RR_+^3$, thus $\{[w^0_n,w^1_n,w^2_n]\}_{n\in\NN}$ is 
also dense in $P\RR_+^3$. Now by the definition given in $\S$\ref{subsec:Teichtheory} for convergence in the Thurston compactification, the fact that $\{[w^0_n,w^1_n,w^2_n]\}_{n\in\NN}$ is dense in $P\RR_+^3$ and that the limit set is closed imply that every point in the simplex is in the limit set of ${\bf r}$.
\medskip
 
We proceed by showing (\ref{Eq:ratio}). As before denote the balanced time of $\alpha^i_n$ along ${\bf r}$ by $t^i_n$. Let $t_n$ be any number in the interval 
 $\Big[\underset{ i=0,1,2 }\min\,t_{2n}^i, \underset{i=0,1,2}\max \,t_{2n}^i\Big]$. 

  The point of choosing such $t_n$ is that (see \figref{intervals}) as we will see below, all three curves 
  $\alpha_{2n}^i, i\in \ZZ_3$ are very short on $X_{t_n}$. Moreover, their collars are asymptotically 
  of the same width.  

  For any $i,j\in \ZZ_3$ by \lemref{shortcurves} and Lemma \ref{Lem:slopes}(\ref{eq : q_2k+1}) we have
  \begin{eqnarray*}
  |t_{2n}^i-t_n|\leq \max_{i,j=0,1,2} |t_{2n}^i-t_{2n}^j| &\eadd& \frac{1}{2}\max_{i,j=0,1,2} |\log q_{2n}^i-\log q_{2n}^j|\\
  &=& \frac{1}{2}\max_{i,j=0,1,2} |\log(q_{2n}^i/q_{2n}^j)|.
  \end{eqnarray*}
  Moreover, by Lemma \ref{Lem:slopes}(\ref{eq : q_2k}), for any $i,j\in\ZZ_3$, we have $q^i_{2n}/q^j_{2n}\sim u^i_n/u^j_n$, 
  and hence $|\log q^i_{2n}/ q^j_{2n}-\log u^i_n/u^j_n |\to 0$ as $n\to\infty$.
   Therefore,
  \[\Big|\max_{i,j=0,1,2} |\log(q_{2n}^i/q_{2n}^j)|-\max_{i,j=0,1,2} |\log(u_{n}^i/u_{n}^j)|\Big|\to 0,\] 
  as $n\to\infty$. Thus for $n$ large enough
  \[2|t_{2n}^i-t_n|\ladd \max_{j=0,1,2} \log(u_{n}^j).\]
Then by Lemma \ref{Lem:slopes} (\ref{eq : log a_2k+1/ log max}) we have that 
  \begin{equation}\label {Eq:close}
    e^{2|t_n-t_{2n}^i|}=o(\log a_{2n+1}^i).
  \end{equation}
  
Now the estimate (\ref{Eq:Hyp}) in Lemma \ref{Lem:shortcurves} and the estimate (\ref{Eq:close}) together with the growth bound (\ref{eq : wolpert}) give us
\[\Hyp_{t_n}(\alpha^i_{2n})\leq \Hyp_{t^i_{2n}}(\alpha^i_{2n}) e^{2|t_n-t^i_{2n}|} \lmul \frac{o(\log a_{2n+1}^i)}{a^i_{2n+1}},\]
also $a^i_{2n+1}\to\infty$ as $n\to\infty$, so the last fraction in the above inequality goes to $0$.
Therefore, for all $n$ sufficiently large, the hyperbolic lengths of the curves
$\alpha_{2n}^i$, $i\in \ZZ_3$, at $X_{t_n}$ are uniformly bounded and in fact very small. 
Also from (\ref{beta_zero}) of \corref{beta_zero} we know that the hyperbolic lengths of  $\beta_i,\, i\in \ZZ_3$, are also uniformly bounded along $\bf r$.

\begin{figure}
    \begin{center}
    \centerline{\includegraphics{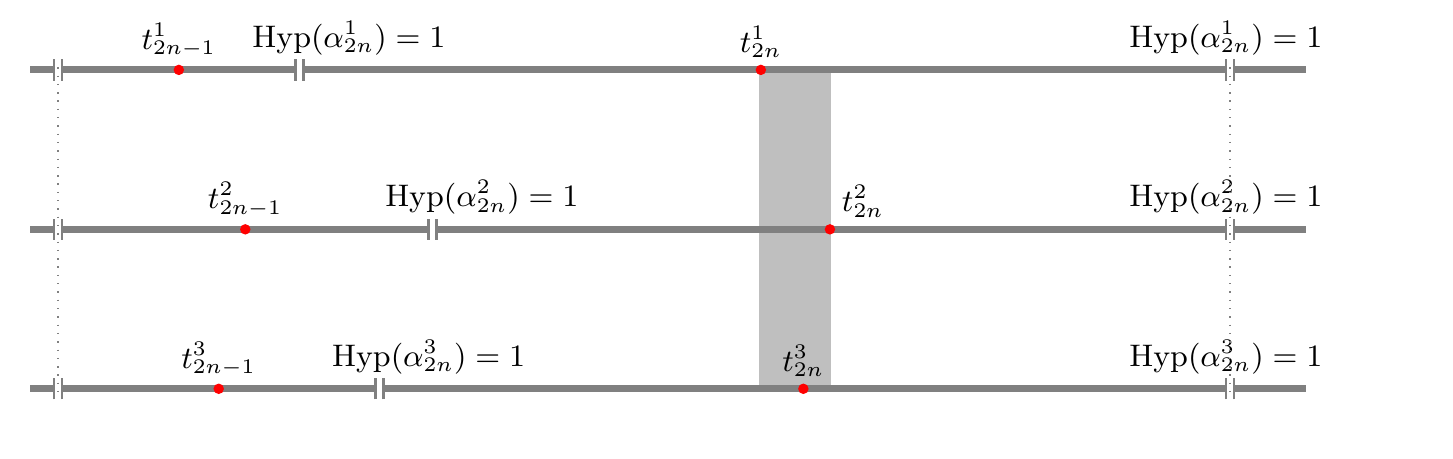}}
    \end{center}
    \caption{The interval when $\alpha_k^i$ is short. For $k=2n$, the curves $\alpha_k^i,\; i\in\ZZ_3$, start getting short at different times, 
    but they grow back to length 1 roughly at the same time. We choose $t_n$ in the shaded interval to guarantee that all three curves are short and have collar neighborhoods of approximately 
    the same width.}
     \label{Fig:intervals}
    \end{figure} 
 
Thus the collection of curves $\{\alpha^i_{2n},\beta^i\}_{i\in\ZZ_3}$ forms a bounded length pants decomposition at time $t_{n}$. Then by (\ref{eq : length contribution}) we have the following estimate for the hyperbolic length of an arbitrary curve $\gamma$ on $X_{t_n}$: 
\begin{align}\label{Eq:length}
   \Hyp_{t_n}(\gamma)=&\sum_{i=0}^2 \Hyp_{t_n}(\gamma,\alpha^i_{2n})
    +\sum_{i=0}^2 \Hyp_{t_n}(\gamma,\beta^i)\\
    +&O\left(\sum_{i=0}^2\I(\gamma,\alpha_{2n}^i)+\I(\gamma,\beta^i)\right), \nonumber
\end{align}
where the constant of $O$ notation depends only on an upper bound for the hyperbolic length of the curves $\alpha^i_{2n},\beta^i,\; i\in\ZZ_3,$ at time $t_{2n}$.
We will now analyze the ingredients of this equation. 

\subsection*{Intersection numbers}
 Note that, for any fixed curve $\gamma$, 
the intersection number with $\beta^i$  is clearly a constant, i.e.
\begin{equation}\label{eq : int beta}\I(\beta^i,\gamma)\eadd 0. \end{equation}
By Lemma \ref{Lem:conv} we have $\frac{\alpha^i_{2n}}{q^i_{2n}\sqrt{1+(\theta^i)^2}}\to\nu^i$, 
hence,  by continuity of intersection numbers \cite{bonahon:TC}, we have
\[
\I\left( \gamma,\frac{\alpha^i_{2n}}{q^i_{2n}\sqrt{1+(\theta^i)^2}} \right)\to \I(\gamma,\nu^i).
\]
Thus
\begin{equation}\label{eq : int alphai2n} 
\I(\gamma,\alpha^i_{2n} ) \emul q^i_{2n}.
\end{equation}
\subsection*{Contribution to the length of $\gamma$ from the curves $\alpha^i_{2n}$ at $t_n$}

First,  the hyperbolic length of $\alpha_{2n}^i$ by inequality (\ref{eq : wolpert}) and the inequality (\ref{Eq:Hyp}) from Lemma \ref{Lem:shortcurves} satisfies
 \[
\frac{1}{a_{2n+1}^i}e^{-2|t_n-t_{2n}^i|}\lmul \Hyp_{t_n}(\alpha_{2n}^i)\lmul \frac{1}{a_{2n+1}^i}e^{2|t_n-t_{2n}^i|},
\]
which using the fact that by \eqnref{width}, $\width_{t^i_{2n}}(\alpha^i_{2n})\eadd -2\log\left(\Hyp_{t^i_{2n}}(\alpha^i_{2n})\right)$ implies the following estimate
\[ 
\width_{t_n}(\alpha_{2n}^i)\eadd 2\log{a_{2n+1}^i}\pm O(2|t_n-t_{2n}^i|).
\]
Then, by \eqnref{close}, \lemref{slopes} (\ref{eq : log a_2k+1/ log max}) and the fact that $a^i_{2n+1}\to \infty$
we deduce that the widths of the collars of the curves $\alpha_{2n}^i$, $i\in\ZZ_3$, 
are equivalent, and that 
\begin{equation}\label{Eq:collars}
\width_{t_n}(\alpha_{2n}^i)\sim 2\log{a_{2n+1}^i}.
\end{equation}
By (\ref{Eq:tw}) in \lemref{shortcurves} and the formula (\ref{eq : twist}) for twist parameters along Teichm\"{u}ller geodesics we have
\[
 \twist_{\alpha_{2n}^i}(\gamma,X_{t_n})\lmul a_{2n+1}^i,
 \]
using \eqnref{close}  then we have
 \begin{equation}\label{Eq:twisting}
 \Hyp_{t_n}(\alpha_{2n}^i)\twist_{\alpha_{2n}^i}(\gamma,X_{t_n})\lmul e^{2|t_n-t_{2n}^i|}=o(\log a_{2n+1}^i). 
 \end{equation}
 Now by (\ref{Eq:collars}) and (\ref{Eq:twisting}) the contribution of $\alpha_{2n}^i$ to the length of $\gamma$ satisfies
 \begin{equation}\label{eq : cont alpha}
 \Hyp_{t_n}(\gamma,\alpha^i_{2n})\sim 2\I(\gamma,\alpha^i_{2n})\log{a_{2n+1}^i}.
 \end{equation}

 \subsection*{Contribution to the length of $\gamma$ from the curves $\beta^i$ at $t_n$} 

From  (\ref{Eq:close}) we have $2|t_n-t^i_{2n}|=o(\log\log a^i_{2n+1})$. Moreover
by \eqnref{balanced} in \lemref{shortcurves} we have $t^i_{2n+1}-t^i_{2n}\gadd \log a^i_{2n+1}$.
Therefore, we have $t_n<t_{2n+1}^i$ for all $n$ sufficiently large.

Now applying  (\ref{beta_decreasing}) and \lemref{dividing}
we get
\begin{equation}
\label{eq : length betai}
\Hyp_{t_n}(\beta^i)\gmul \frac{1}{\log{q_{2n+1}^i}},
\end{equation}
which by  the fact that $\width_{t_n}(\beta^i)\eadd-2\log(\Hyp_{t_n}(\beta^i))$ implies that 
\begin{equation}\label{eq : w tin}
\width_{t_n}(\beta^i)\ladd 2\log\log{q_{2n+1}^i}.
\end{equation}

Moreover, by (\ref{Eq:qprod}) in Lemma \ref{Lem:slopes} we have 
$\log\log{q_{2n+1}^i}\eadd \log (\sum_{j=1}^{2n+1} \log a^i_j)$.
Now conditions (\ref{eq : q1}) and (\ref{eq : q3}) from the setup of the continued fractions in $\S$\ref{subsec : cf setup}
imply that for each $i\in\ZZ_3$ the sequence $\{a^i_j\}_j$ is increasing, in fact it is increasing at least exponentially fast. Hence
\begin{eqnarray*}
\log (\sum_{j=1}^{2n+1} \log a^i_j)&\leq& \log((2n+1)\log{a^i_{2n}})\\
&=&\log{(2n+1)}+\log\log a^i_{2n+1}\emul\log\log a^i_{2n+1}.
\end{eqnarray*}
We then have that
\begin{equation}\label{Eq:bcollars}
\width_{t_n}(\beta^i)\lmul\log\log{a_{2n+1}^i}.
\end{equation}

Moreover, since $\beta^i$ is a union of critical trajectories, it does not have a flat cylinder neighborhood. 
Therefore, $\I_{\beta^i}(\nu^-,\nu^+)\eadd 0$. Then, by (\ref{eq : twist}), we have

\begin{equation}\label{eq : tw betai tn}
\Hyp_{t_n}(\beta^i)\twist_{\beta^i}(\gamma,X_{t_n})\leq K_\gamma,
\end{equation}
where $K_\gamma\geq 0$ depends only on $\gamma$. 

Hence by equations (\ref{eq : int beta}), (\ref{Eq:bcollars}), (\ref{eq : tw betai tn}) and \corref{beta_zero}, the contribution to 
the length of $\gamma$ from the curve $\beta^i$  for $i\in\ZZ_3$ at time $t_n$ satisfies
\begin{eqnarray}\label{eq : cont beta tn} 
\Hyp_{t_n}(\gamma,\beta^i)&=&\I(\gamma,\beta^i)\left( \width_{t_n}(\beta^i)+\Hyp_{t_n}(\beta^i)\twist_{\beta^i}(\gamma,X_{t_n}) \right)\nonumber\\
&\lmul &\I(\gamma,\beta^i)(\log\log a^i_{2n+1}+K_\gamma)=o(\log a^i_{2n+1}).
\end{eqnarray}

 We are now ready to establish \eqnref{ratio}. First, we use \eqnref{length} 
and equations (\ref{eq : int beta}), (\ref{eq : cont alpha}) and (\ref{eq : cont beta tn}) for the curves $\gamma_1$ and $\gamma_2$ to get
\begin{equation}\label{Eq:almost}
 \frac{\Hyp_{t_n}(\gamma_1)}{\Hyp_{t_n}(\gamma_2)}
 \sim\frac{\sum_{i=0}^2 \I(\gamma_1,\alpha_{2n}^i)\log{a_{2n+1}^i}}
 {\sum_{i=0}^2 \I(\gamma_2,\alpha_{2n}^i)\log{a_{2n+1}^i}}\sim \frac{\sum_{i=0}^2 \I(\gamma_1,\alpha_{2n}^i)}
 {\sum_{i=0}^2 \I(\gamma_2,\alpha_{2n}^i)},
\end{equation}
where the second comparison holds by \eqnref{aequiv} in Lemma \ref{Lem:slopes}. 
Then \corref{ratios} applied to \eqnref{almost} gives us the desired \eqnref{ratio}.
\medskip

As we saw above the limit set of ${\bf r}$ contains the simplex of projective measures spanned by $[\nu^i]$, $i\in \ZZ_3$.
 To complete the proof of the theorem it remains to show that the limit set of ${\bf r}$ is also contained in the simplex. First, note that any limit point of ${\bf r}$
has zero intersection number with the vertical measured foliation $\nu$ of $\phi_0$ which is the disjoint union of foliations $\nu^i$ and curves $\beta^i,$ $\; i\in\ZZ_3$. 
Hence all we need to show is that every point in the limit set has zero weight on $\beta^i, \,i\in \ZZ_3$. 

For this purpose suppose that for a sequence of times $\{t_k\}_k$ the sequence  $\{{\bf r}(t_k)\}_k$ converges to the projective class of some measured foliation $\mu$  in $\PMF(S)$. Then as is shown in \cite[expo\'se 8]{flp:TTs} there is a sequence $\{s_k\}_k$ with $s_k\to 0$, so that  for any simple closed curve $\gamma$ we have
\begin{equation} \label{eq:convThbdry}
\lim_{k\to\infty}s_k\Hyp_{t_k}(\gamma)= \I(\gamma, \mu).
\end{equation}  
 To show that $\mu$ has zero weight on $\beta^i$ for all $i\in \ZZ_3$, we argue as follows.
 
Given $i\in\ZZ_3$ let $\gamma$ be any simple closed curve  that intersects $\beta_i$  twice and does not intersect any $\beta_j$ with $j\neq i$.
Let $\gamma'\subset T_i$ be a simple closed curve obtained from the concatenation of the arc $\gamma\cap T_i$ and a sub-arc of the boundary of $T_i$. 
That $\mu$ has zero weight on $\beta^i$ follows from
\begin{equation} \label{Eq:gammas}
\lim_{k\to\infty}\frac{\Hyp_{t_k}(\gamma)}{\Hyp_{t_k}(\gamma')}= 1.
\end{equation}
Indeed,  the above limit and (\ref{eq:convThbdry}) together imply that $\I(\gamma,\mu)=\I(\gamma',\mu)$. Let $\mu=\sum_{j=0}^{2}a_j\nu^j+\sum_{i=0}^{2}b_j\beta^j$. 
By the choice of $\gamma'$ and $\gamma$ we have that $\I(\mu,\gamma')=a_i\I(\nu^i,\gamma')$ and $\I(\mu,\gamma)=a_i\I(\nu^i,\gamma)+b_i\I(\beta^i,\gamma).$ 
Since we also have that $\I(\nu^i,\gamma)=\I(\nu^i,\gamma')$, we see that $b_i=0$.
  
To prove (\ref{Eq:gammas}), we first use a surgery argument  and \ref{eq : length contribution} to obtain for any $t>0$
\[\Hyp_t(\gamma')-\Hyp_t(\beta^i)\leq \Hyp_t(\gamma)\leq \Hyp_t(\gamma')+\Hyp_t(\gamma,\beta^i)+A_\gamma
\]
where $A_\gamma$ depends on $\gamma$ only. Since $\Hyp_t(\gamma')\to \infty$ (see Claim \ref{claim : len gamma alphai}) and $\Hyp_t(\beta^i)\to 0$ (\ref{beta_zero}),  if we show that $\lim_{t\to +\infty}\frac{\Hyp_t(\gamma,\beta)}{\Hyp_t(\gamma')}\to 0$, this will imply \eqnref{gammas}. 

Let  $t\geq 0$ and  let $\alpha^i=\alpha^i(t)$ be a shortest curve in $T^i$ with respect to the flat metric at time $t$. Then we have (see, for example Proposition 3.1 in \cite{rafi:BC})
\[ \Hyp_{t}(\gamma')\emul  \Hyp_{t}(\gamma',\alpha^i).\]
Now again by the choice of the curves $\gamma$ and $\gamma'$,  $\Hyp_{t}(\gamma,\alpha^i)\eadd \Hyp_{t}(\gamma',\alpha^i)$, which 
implies that it suffices to prove that

\begin{equation}\label{Eq:nobeta}
\lim_{t\to\infty}\frac{\Hyp_{t}(\gamma,\beta^i)}{\Hyp_{t}(\gamma,\alpha^i)}= 0,
\end{equation}



Thus to complete the proof of the theorem it suffices to prove (\ref{Eq:nobeta}).

From \cite[Lemma 1]{lenzhen:GL} we know that $\alpha^i=\alpha^i_n$ for some $n=n(t)\geq 1$, where $n\to +\infty$ with $t\to +\infty$. Let $0\leq \us_n^i\leq \bs_n^i$ be such that $\ell_{\us_n^i}(\alpha_n^i)=\ell_{\bs_n^i}(\alpha_n^i)=2$. 
Note that any flat torus of area $1$ and with a slit contains a simple closed curve of length at most $2$, provided that the slit is small. 
Then since $\alpha^i_n$ is a shortest curve contained in $T_i$ at time $t$, we
see that the interval $[\us_n^i,\bs_n^i]$ is not empty and contains $t$. We also have the balanced time $t_n^i\in [\us_n^i,\bs_n^i]$, which is  the midpoint
of this interval. The following claim holds for any simple closed curve $\gamma$ such that $\I(\gamma,\beta^i)\neq 0$, although we will only use it for the $\gamma$ defined above. 

\begin{claim}\label{claim : len gamma alphai} We have the following estimate for the contribution of $\alpha_n^i$ to the length of $\gamma$ at any time $t$ large enough:
\begin{equation}\label{Eq:gamma_alphai}
\Hyp_t(\gamma, \alpha_n^i) \geq  B_\gamma \begin{cases}e^{\us_n^i} (1+(t-\us_n^i)) & \text{if $t\in [\us_n^i, t^i_n]$},\\ 
e^{\us_n^i}(1+(\bs_n^i-t)+e^{2(t-t^i_n)}) & \text{if $t\in [t^i_n,\bs_n^i]$}.
\end{cases}
\end{equation}
In particular, $\Hyp_t(\gamma, \alpha_n^i) \geq B_\gamma t$ for all $t\in [\us_n^i,\bs_n^i]$. Here the constant  $B_\gamma$ depends only on $\gamma$.
\end{claim}
\begin{proof} 
Recall that 
$$\Hyp_t(\gamma, \alpha_n^i)= \I(\gamma, \alpha_n^i) \left(\width_t(\alpha_n^i)+\Hyp_{t^i_n}(\alpha_n^i)\twist_{\alpha_n^i}(\gamma,X_t)\right).$$
We first compute the times $\bs_n^i$ and $\us_n^i$. By Equation (3.8) in Lemma \ref{Lem:shortcurves}, $\ell_{t^i_n}(\alpha^i_n)\emul \frac{1}{\sqrt{a^i_{n+1}}}$, then since $\ell_{\us_n^i}(\alpha_n^i)\emul \ell_{t_n^i}(\alpha_n^i)e^{|t_n^i-\us_n^i|}$ (see e.g. the discussion before  Equation (2) in \cite{rafi:HT}), we have that
\begin{equation}\label{eq:t-us}
t_n^i-\us_n^i \eadd \frac{1}{2}\log{ a_{n+1}^i}.
\end{equation}
Similarly, we have that
\begin{equation}\label{eq:bs-t}
 \bs_n^i-t_n^i\eadd \frac{1}{2}\log{ a_{n+1}^i}.
 \end{equation}
Hence from \eqnref{balanced} we obtain 
\begin{equation} \label{Eq:active}\us_n^i\eadd \sum_{j=1}^n \log{a_j^i}\,\,\,\text{ and }\,\,\, \bs_n^i\eadd \sum_{j=1}^{n+1} \log{a_j^i}.
\end{equation}
Next thing to note is that since at $\us^i$ and $\bs^i$ the curve $\alpha_n^i$ has length 2 in the flat metric, it follows from \cite[Theorem 6]{Rafi:tt} that 

 \[\Hyp_{\us_n^i}(\alpha_n^i)\emul 1\; \text{and}\; \Hyp_{\bs_n^i}(\alpha_n^i)\emul 1.\] 
 Also, by Lemma  \ref{Lem:shortcurves}, $\Hyp_{t_n^i}(\alpha_n^i)\emul \frac{1}{a_{n+1}^i}$, 
so by (\ref{eq : wolpert}) for any $t\in [\us_n^i,\bs_n^i]$ we have
\begin{equation}\label{Eq:lalphaa}
\Hyp_t(\alpha_n^i)\emul \frac{e^{2|t-t_n^i|}}{a_{n+1}^i}.
\end{equation}
Since by (\ref{eq:t-us}) and (\ref{eq:bs-t}), $a^i_{n+1}\emul e^{2(\bs_n^i-t_n^i)}=e^{2(t_n^i-\us_n^i)}$ we can rewrite the above coarse equality as
\begin{equation}\label{Eq:lalpha}
\Hyp_t(\alpha_n^i)\emul 
\begin{cases}  e^{2(\us_n^i-t)} & \text{if } t\in[\us_n^i,t_n^i]\\
              e^{2(t-\bs_n^i)}       & \text{if } t\in[t_n^i,\bs_n^i].
\end{cases}
\end{equation}
 For $t\in [\us_n^i,t_n^i]$, the size of the collar $\width_t(\alpha_n^i)$ by \eqnref{width} and \eqnref{lalpha} is bounded below by
 $$w(t)=2\arcsinh{\frac{1}{\sinh{\frac A 2 e^{2(\us_n^i-t)}}}}$$
 where $A>1$ is a multiplicative error in \eqnref{lalpha}.
 By a straightforward computation  $w'(t)$ is increasing on $[\us_n^i,t_n^i]$, so we have
 $$w(t)\geq w'(\us_n^i)(t-\us_n^i)+w(\us_n^i).$$ 
 Hence for the $t\in [\us_n^i,t_n^i]$ we have
  $$\width_t(\alpha_n^i)\geq  \frac{2A}{\sinh{\frac A 2}}(t-\us_n^i)+2\arcsinh{\frac{1}{\sinh{\frac A 2 }}}$$
  which we write simply as 
  \begin{equation}\label{Eq:widthgrow}
  \width_t(\alpha_n^i)\gmul (t-\us_n^i)+1. 
  \end{equation}
 By a similar argument  for any $t\in[t_n^i,\bs_n^i]$ we have that
 \begin{equation}\label{Eq:widthdecay}
  \width_t(\alpha_n^i)\gmul (\bs_n^i-t)+1. 
  \end{equation}
 
  Let the slope of $\gamma$ in $T^i$ be $\frac{a}{b}$ and recall that the slope of $\alpha^i_n$ is $\frac{p^i_n}{q^i_n}$. 
 Then $\I(\gamma,\alpha^i_n)=|q^i_na-p^i_nb|=q^i_n|a-b\frac{p^i_n}{q^i_n}|$ 
 and since $\frac{p^i_n}{q^i_n}$ converges to $\theta^i$, the slope of $\nu^i$, 
 we see that $\I(\gamma,\alpha^i_n)$ is $q^i_n$ up to a multiplicative error that depends only on $\gamma$. 
 Therefore, from (\ref{Eq:qprod}) in Lemma \ref{Lem:slopes} and \eqnref{active}
 we have for some $C_\gamma$
\begin{equation}\label{Eq:igamma}
\frac{1}{C_\gamma} e^{\us_n^i}\leq \I(\gamma, \alpha_n^i)\leq C_\gamma e^{\us_n^i}.
\end{equation}
Hence for any $t\in[\us_n^i,t_n^i]$, applying \eqnref{widthgrow}, \eqnref{igamma}, we have for some $D_\gamma>0$ that only depends on $\gamma$ such that
  \[\Hyp_t(\gamma, \alpha_n^i)\geq D_\gamma e^{\us_n^i}((\bs_n^i-t)+1+e^{2(t-t_n^i)}).\]
 
 Further, for  $t\in[t_n^i,\bs_n^i]$ the collar about $\alpha^i$ is shrinking, so we need to add information about the twisting. From Equation (\ref{eq : twist}),  \eqnref{tw} and  \eqnref{lalphaa} we have the inequality
 \begin{equation}
\Hyp_t(\alpha_n^i)\twist_{\alpha_n^i}(\gamma,X_t)\geq B e^{2(t-t_n^i)}-E_\gamma,
 \end{equation}
 where $E_\gamma$ depends only on $\gamma$. This estimate together with \eqnref{igamma} and \eqnref{widthdecay} imply that there is $F_\gamma>0$ such that for the $t\in[t_n^i,\bs_n^i]$ the coarse inequality
 \[\Hyp_t(\gamma, \alpha_n^i)\geq F_\gamma e^{\us_n^i}((\bs_n^i-t)+1+e^{2(t-t_n^i)})\] holds.
 Here, if we let $t$ be large enough that $\bs^i_n-t^i_n\geq 2E_\gamma$, then either
$t-t^i_n\geq E_\gamma$ or $\bs^i_n-t\geq E_\gamma$, and hence we may absorb the constant $E_\gamma$ in the multiplicative constant. 
Letting  $B_\gamma=\min\{D_\gamma, F_\gamma\}$  completes the proof of the claim.
\end{proof}

Now we estimate the contribution form $\beta^i$ to the length of $\gamma$ at time $t$.
 The curve $\beta^i$ is a vertical curve, that is a union of critical trajectories, hence $\ell_t(\beta^i) \lmul e^{-t}$. Then, since
$\beta^i$ does not have a flat cylinder neighborhood, applying Equation (\ref{eq : ext length modulus of annuli}) and the estimates for the moduli of annular neighborhoods of $\beta^i$ before the equation we have that $\Ext_t(\beta^i)\lmul \frac{1}{t}$. 
Then for $t\gg 0$  by (\ref{eq:Maskit2}) we obtain 
\[\Hyp_t(\beta^i) \lmul \frac{1}{t},\]
and hence by \eqnref{width} we have
\begin{equation}\label{eq:width betai}
\width_{t}(\beta^i) \ladd \log t.
\end{equation}

Also, $\I_{\beta^i}(\nu^-,\nu^+)\eadd 0$ and hence by (\ref{eq : twist}) for all $t\geq 0$ we have 
\begin{equation}\label{Eq:twbetai}
\Hyp_t(\beta^i)\twist_{\beta^i}(\gamma,X_{t})\le a_\gamma,
\end{equation}
for a constant $a_\gamma$ depending only on $\gamma$. Therefore, for $t$ large enough we have
\begin{align}\label{eq : cont beta} 
\Hyp_t(\gamma,\beta^i) 
&= \I(\gamma,\beta^i) \left( \width_t(\beta^i)
    +\Hyp_t(\beta^i)\twist_{\beta^i} (\gamma,X_t) \right) \notag \\
&\leq   b_\gamma\log t,
\end{align}
where  $b_\gamma$ depends on $\gamma$ only. 

The coarse inequality (\ref{eq : cont beta}) and the Claim \ref{claim : len gamma alphai} give us \eqnref{nobeta}, which completes the proof of our theorem.
\end{proof}
 \bibliographystyle{alpha}
  \bibliography{main}
\end{document}